\documentclass[11pt]{amsart}
\usepackage{amsmath,amssymb,amsthm}
\usepackage{mathtools}
\usepackage{enumitem}
\usepackage{booktabs}
\usepackage{hyperref}

\hypersetup{hidelinks}

\newtheorem{theorem}{Theorem}[section]
\newtheorem{proposition}[theorem]{Proposition}
\newtheorem{lemma}[theorem]{Lemma}

\theoremstyle{remark}
\newtheorem{remark}[theorem]{Remark}
\newtheorem{conjecture}[theorem]{Conjecture}

\newcommand{\Q}{\mathbb{Q}}
\newcommand{\Z}{\mathbb{Z}}

\newcommand{\sq}{\mathord{\Box}}

\title[Elliptic curves induced by Fibonacci triples]{Ranks and integer points on elliptic curves induced by Fibonacci triples}
\author{Andrej Dujella}
\date{}

\begin{document}

\dedicatory{Dedicated to the memory of my dear friend and coauthor Florian Luca}

\begin{abstract}
Let $F_n$ and $L_n$ denote the Fibonacci and Lucas numbers, respectively, and consider
\[
 E_k:\qquad y^2=(F_{2k}x+1)(F_{2k+2}x+1)(F_{2k+4}x+1).
\]
These elliptic curves arise naturally from the regular Diophantine triples
\[
 \{F_{2k},F_{2k+2},F_{2k+4}\}.
\]
For odd $k$, we exhibit the rational point
\[
 Q_k=\left(
 -\frac{F_{k-1}}{L_kF_{k+1}F_{k+2}},
 \frac{F_{2k+1}}{L_kF_{k+1}F_{k+2}}
 \right).
\]
For every odd $k\geq 3$, this point is independent of the standard point
$P_k=(0,1)$; in particular, $\operatorname{rank}E_k(\Q)\geq 2$.
Moreover, if $k\geq 3$ is odd and $\operatorname{rank}E_k(\Q)=2$, then all
integer points on $E_k$ are exactly the points arising from the two known
solutions of the Hoggatt-Bergum extension problem.  By parametrizing the
two conics $L^2-5F^2=\pm4$ and applying an injective specialization
criterion, we also show that the corresponding one-parameter elliptic
families have generic ranks $2$ in the odd case and $1$ in the even case.
Finally, we discuss computational data and propose the heuristic rank
distribution $1/4,1/2,1/4$ for ranks $1,2,3$, respectively, with density
zero for rank at least $4$.
\end{abstract}

\maketitle

\section{Introduction}

A set of positive integers $\{a_1,\ldots,a_m\}$ is called a Diophantine
$m$-tuple if $a_i a_j+1$ is a perfect square for all $i\neq j$.
The Fibonacci numbers satisfy
\[
 F_{2k}F_{2k+2}+1=F_{2k+1}^2,\qquad
 F_{2k}F_{2k+4}+1=F_{2k+2}^2,
\]
and
\[
 F_{2k+2}F_{2k+4}+1=F_{2k+3}^2.
\]
Hence
\[
 \{F_{2k},F_{2k+2},F_{2k+4}\}
\]
is a Diophantine triple for every positive integer $k$.  Hoggatt and
Bergum~\cite{HoggattBergum} observed that it can be extended to the Diophantine quadruple
\[
 \{F_{2k},F_{2k+2},F_{2k+4},
 4F_{2k+1}F_{2k+2}F_{2k+3}\},
\]
and conjectured that the fourth element is unique.  This conjecture was
proved in \cite{DujellaHB}. A simple consequence of this result is that 
if $\{F_{2k}, F_{2k+2}, F_{2k+4}, d\}$ is a Diophantine quadruple, then $d$ cannot be a Fibonacci number 
(see \cite[Section 4.9.1]{DujellaBook}). 

Several authors considered the question of how large Diophantine tuples consisting of Fibonacci numbers can be.
He, Luca, and Togb\'e~\cite{HLT} proved that if $\{F_{2n}, F_{2n+2}, F_k\}$ is a Diophantine triple, 
then $k = 2n+4$ or $k = 2n - 2$ (when $n > 1$), except when $n = 2$, in which case $k = 1$ is also possible. 
Fujita and Luca~\cite{FL} proved that there are only finitely many Diophantine quadruples consisting of Fibonacci numbers, and in~\cite{FL2} they proved that there are no such quadruples.

To the triple $\{F_{2k},F_{2k+2},F_{2k+4}\}$ one associates the elliptic curve
\begin{equation}\label{Ek}
 E_k:\qquad
 y^2=(F_{2k}x+1)(F_{2k+2}x+1)(F_{2k+4}x+1).
\end{equation} 
We say that $E_k$ is induced by the triple $\{F_{2k},F_{2k+2},F_{2k+4}\}$.  
The curves $E_k$ were studied in \cite{DujellaMT} and, more recently, in
\cite[Sections 4.8-4.9]{DujellaBook}.  It was proved that
\[
 E_k(\Q)_{\rm tors}\simeq \Z/2\Z\times\Z/2\Z
\]
and that $\operatorname{rank}E_k(\Q)\geq 1$.  Under the assumption that 
$\operatorname{rank}E_k(\Q)=1$, all integer points were determined.
In particular, for $k\geq2$ they are
\begin{equation}\label{knownpoints}
 (0,\pm1)
\end{equation}
and
\begin{equation}\label{HBpoint}
 \left(
 d_k,\,
 \pm(2F_{2k+1}F_{2k+2}-1)
 (2F_{2k+2}^2+1)
 (2F_{2k+2}F_{2k+3}+1)
 \right),
\end{equation}
where
\[
 d_k=4F_{2k+1}F_{2k+2}F_{2k+3}.
\]

Computations with PARI/GP suggested an unusual rank distribution in this
family: among the curves of odd expected rank, curves of expected rank
$3$ occur almost as frequently as curves of expected rank $1$; see
\cite[Remark 4.9.5]{DujellaBook}.  The purpose of this note is to explain
this phenomenon.  The main new ingredient is a second systematic rational
point which appears precisely in the odd-index subfamily.

Our main result is the following.

\begin{theorem}\label{main}
Let $k\geq3$ be odd. Then:
\begin{enumerate}[label={\rm (\roman*)}]
\item the point
\[
 Q_k=\left(
 -\frac{F_{k-1}}{L_kF_{k+1}F_{k+2}},
 \frac{F_{2k+1}}{L_kF_{k+1}F_{k+2}}
 \right)
\]
belongs to $E_k(\Q)$;
\item the points $P_k=(0,1)$ and $Q_k$ are independent, and therefore
\[
 \operatorname{rank}E_k(\Q)\geq2;
\]
\item if $\operatorname{rank}E_k(\Q)=2$, then the integer points on
$E_k$ are precisely the points in \eqref{knownpoints} and
\eqref{HBpoint}.
\end{enumerate}
\end{theorem}

For $k=1$ the same formula gives $Q_1=(0,1)=P_1$, so the restriction
$k\geq3$ in parts (ii) and (iii) is necessary.

\section{The additional point}

We use the standard notation
\[
 F_0=0,\quad F_1=1,\quad F_{n+2}=F_{n+1}+F_n,
\]
and
\[
 L_0=2,\quad L_1=1,\quad L_{n+2}=L_{n+1}+L_n.
\]
We shall repeatedly use
\[
 L_n=F_{n-1}+F_{n+1},\qquad F_{2n}=F_nL_n,
\]
\[
 F_{2n+1}=F_n^2+F_{n+1}^2,
\]
and Cassini's identity
\[
 F_{n+1}F_{n-1}-F_n^2=(-1)^n.
\]

For odd $k$, put
\[
 f=F_k,\qquad u=F_{k+1},\qquad v=F_{k+2},
\]
\[
 \ell=L_k,\qquad \ell_1=L_{k+1},\qquad \ell_2=L_{k+2},
 \qquad r=F_{2k+1}.
\]
We shall also write
\[
 a=F_{2k},\qquad b=F_{2k+2},\qquad c=F_{2k+4}.
\]
Notice that
\begin{equation}\label{regular}
 b-a=r,\qquad c=a+b+2r,
\end{equation}
so the Fibonacci triple is a regular Diophantine triple (see e.g. \cite[Section 1.1]{DujellaBook}).

The following three identities will be useful:
\begin{align}
 uv-fF_{k-1}&=r,\label{id1}\\
 \ell v-\ell_1F_{k-1}&=r,\label{id2}\\
 \ell u-\ell_2F_{k-1}&=1.\label{id3}
\end{align}
They follow immediately from the standard addition formulas for Fibonacci
and Lucas numbers (or directly from Cassini's identity and the recurrence).

\begin{proposition}\label{Qpoint}
Let $k$ be odd and let
\[
 D_k=L_kF_{k+1}F_{k+2}.
\]
Then
\[
 Q_k=\left(-\frac{F_{k-1}}{D_k},
           \frac{F_{2k+1}}{D_k}\right)
\]
is a rational point on $E_k$.
\end{proposition}

\begin{proof}
Put $q=-F_{k-1}/D_k$.  Since $a=f\ell$, $b=u\ell_1$ and
$c=v\ell_2$, identities \eqref{id1}-\eqref{id3} give
\begin{align}
 aq+1&=\frac{r}{uv},\label{Q1}\\
 bq+1&=\frac{r}{v\ell},\label{Q2}\\
 cq+1&=\frac{1}{u\ell}.\label{Q3}
\end{align}
Therefore
\[
 (aq+1)(bq+1)(cq+1)
 =\frac{r^2}{\ell^2u^2v^2}
 =\left(\frac{r}{D_k}\right)^2,
\]
which proves the assertion.
\end{proof}

For example,
\[
 Q_3=\left(-\frac1{60},\frac{13}{60}\right),\qquad
 Q_5=\left(-\frac3{1144},\frac{89}{1144}\right).
\]

\section{Independence of the two points}

We follow the $2$-descent approach used in
\cite{DujellaMT,DujellaPetho,DujellaParam}. 
Applying $X = abcx$ and $Y = abcy$ transforms \eqref{Ek} into
\begin{equation}\label{E0}
 E_k':\qquad
 Y^2=(X+bc)(X+ac)(X+ab). 
\end{equation}
The three non-trivial points of order $2$ are
\[
 A=(-bc,0),\qquad B=(-ac,0),\qquad C=(-ab,0),
\]
and $P_k=(0,1)$ corresponds to
\[
 P=(0,abc).
\]
We use the same letters for a point on $E_k$ and its image on $E_k'$ when
no confusion can arise.

For $R=(X,Y)\in E_k'(\Q)$, the usual $2$-descent maps are represented,
away from the corresponding $2$-torsion point, by the square classes of
\[
 X+bc,\qquad X+ac,\qquad X+ab.
\]
In particular, $R\in2E_k'(\Q)$ if and only if all three corresponding
classes are trivial, with the standard special definitions at the
$2$-torsion points; cf.\ \cite{DujellaMT}.

We need two elementary coprimality observations.

\begin{lemma}\label{coprime}
Let $k\geq3$ be odd.
\begin{enumerate}[label={\rm (\roman*)}]
\item If an odd prime $p$ divides $v=F_{k+2}$, then
\[
 p\nmid rfu\ell\ell_1\ell_2(c-a)(c-b).
\]
\item If an odd prime $p$ divides $\ell_2=L_{k+2}$, then
\[
 p\nmid rfuv\ell\ell_1(c-a)(c-b).
\]
\end{enumerate}
\end{lemma}

\begin{proof}
Assume first that $p\mid v$.  Since $v=f+u$, we have
$u\equiv-f\pmod p$.  Cassini's identity, in the form
\[
 u^2-fu-f^2=-1,
\]
gives $f^2\equiv-1\pmod p$.  Consequently
\[
 r=f^2+u^2\equiv-2,\qquad
 \ell=2u-f\equiv-3f,
\]
\[
 \ell_1=f+v\equiv f,\qquad
 \ell_2=2u+v\equiv-2f\pmod p.
\]
Furthermore,
\[
 c-a=L_{2k+2}=\ell_1^2-2\equiv-3\pmod p
\]
and
\[
 c-b=F_{2k+3}=u^2+v^2\equiv-1\pmod p.
\]
This proves (i).

Now assume $p\mid\ell_2$.  Since $\ell_2=f+3u$, we have
$f\equiv-3u\pmod p$.  Cassini's identity gives
$5u^2\equiv1\pmod p$.  Hence
\[
 r=f^2+u^2\equiv2,\qquad
 v=f+u\equiv-2u,
\]
\[
 \ell=2u-f\equiv5u,\qquad
 \ell_1=2f+u\equiv-5u\pmod p.
\]
Also
\[
 c-a=L_{2k+2}=\ell_1^2-2\equiv3\pmod p,
\]
and
\[
 c-b=F_{2k+3}=u^2+v^2\equiv1\pmod p.
\]
This proves (ii).
\end{proof}

We shall use the classical theorem of Cohn \cite{Cohn}: the only square
Fibonacci numbers are $F_1=F_2=1$ and $F_{12}=144$, while the only
Fibonacci numbers which are twice a square are $F_3=2$ and $F_6=8$.
Similarly, the only square Lucas numbers are $L_1=1$ and $L_3=4$.

We shall also use the elementary parity facts
\begin{equation}\label{2adic}
 v_2(F_n)=
 \begin{cases}
 0,&3\nmid n,\\
 1,&3\mid n,
 \end{cases}
 \qquad
 v_2(L_n)=
 \begin{cases}
 0,&3\nmid n,\\
 2,&3\mid n,
 \end{cases}
 \qquad (n\ {\rm odd}),
\end{equation}
which follow immediately from the Fibonacci and Lucas recurrences modulo
$4$ and $8$, respectively.  Thus, for odd $n$, if every odd prime occurs
to even exponent in $F_n$ (respectively $L_n$), then $F_n$ is a square
or twice a square (respectively $L_n$ is a square).

The following table is obtained by applying the three descent
homomorphisms to $Q$, $Q+P$, $Q+C$ and $Q+P+C$.  Here $\sq$ denotes
the square of a rational number.
\begin{equation}\label{table}
\begin{array}{c|ccc}
R& ax+1&bx+1&cx+1\\ \hline
Q&
ruv\,\sq&
rv\ell\,\sq&
u\ell\,\sq\\[1mm]
Q+P&
r\ell_1\ell_2\,\sq&
rf\ell_2\,\sq&
f\ell_1\,\sq\\[1mm]
Q+C&
rv\ell_1(c-a)\,\sq&
rfv(c-b)\,\sq&
f\ell_1(c-a)(c-b)\,\sq\\[1mm]
Q+P+C&
ru\ell_2(c-a)\,\sq&
r\ell\ell_2(c-b)\,\sq&
u\ell(c-a)(c-b)\,\sq .
\end{array}
\end{equation}
For the four points
\[
 Q+A,\quad Q+B,\quad Q+P+A,\quad Q+P+B,
\]
exactly two of the three corresponding coefficients are negative.

\begin{lemma}\label{newclasses}
Let $k\geq3$ be odd and let
$T\in E_k'(\Q)_{\rm tors}=\{O,A,B,C\}$. Then
\[
 Q+T\notin2E_k'(\Q),\qquad Q+P+T\notin2E_k'(\Q).
\]
\end{lemma}

\begin{proof}
The descent classes of a point $R$ are represented by the coefficients
in the row corresponding to $R+P$ in \eqref{table}, because the descent classes of
$P$ are represented by $bc,ac,ab$.

For $T=A$ or $B$, the assertion follows immediately from the two
negative coefficients mentioned above.

Suppose $T=O$ or $C$.  If $Q+T\in2E_k'(\Q)$, the corresponding row
$Q+P+T$ shows, by Lemma \ref{coprime}(ii), that every odd prime occurs
to an even exponent in $L_{k+2}$.  By \eqref{2adic}, $L_{k+2}$ is a
square, contrary to Cohn's theorem, since $k+2\geq5$ is odd.

Similarly, if $Q+P+T\in2E_k'(\Q)$, the row belonging to $Q+T$ and
Lemma \ref{coprime}(i) imply
\[
 F_{k+2}=\sq\qquad\text{or}\qquad F_{k+2}=2\sq,
\]
again impossible.
\end{proof}

\begin{proposition}\label{independent}
For every odd $k\geq3$, the points $P_k$ and $Q_k$ are independent in
$E_k(\Q)$.
\end{proposition}

\begin{proof}
It was proved in \cite{DujellaMT} that
\[
 E_k(\Q)_{\rm tors}\simeq\Z/2\Z\times\Z/2\Z.
\]
Moreover, \cite[Lemma 2]{DujellaMT} shows that
\[
 P+T\notin2E_k'(\Q)
\]
for every torsion point $T$, except possibly in the exceptional
case in which $c$, $c-a$ and $c-b$ are all twice squares.
Here $c=F_{2k+4}$, and Cohn's theorem excludes this possibility for
$k\geq3$.

Together with Lemma \ref{newclasses}, this shows that no non-zero element
among
\[
 P+T,\qquad Q+T,\qquad P+Q+T,\qquad T\in E_k'(\Q)_{\rm tors},
\]
belongs to $2E_k'(\Q)$.

Suppose that $mP+nQ$ is torsion.  If $m$ and $n$ are not both even,
reduction modulo $2E_k'(\Q)$ contradicts the preceding paragraph.
Hence $m$ and $n$ are both even, say $m=2m_1$ and $n=2n_1$.
Since $2(m_1P+n_1Q)$ is torsion, the point $m_1P+n_1Q$ is torsion as
well.  Repeating the same argument shows that $m$ and $n$ are divisible
by arbitrarily high powers of $2$.  Therefore $m=n=0$.
\end{proof}

Together with Proposition \ref{Qpoint},
this proves parts (i) and (ii) of Theorem \ref{main}.

\section{Integer points when the rank is two}

We now assume throughout this section that $k\geq3$ is odd and
\[
 \operatorname{rank}E_k(\Q)=2.
\]
The proof of Proposition \ref{independent} shows that the images of $P$ and $Q$ 
are linearly independent in $E_k'(\Q)/(2E_k'(\Q) + E_k'(\Q)_{\rm tors}$. 
Since $\operatorname{rank}E'_k(\Q)= 2$, these images form a basis. Therefore, every point of
$E_k'$ is congruent modulo $2E_k'(\Q)$ to one of the sixteen points
\[
 T,\quad P+T,\quad Q+T,\quad P+Q+T,
 \qquad T\in\{O,A,B,C\}.
\]
This is exactly the same enlargement from eight to sixteen
$2$-descent classes which occurs in the rank-$2$ arguments of
\cite{DujellaPetho,DujellaParam}.

For the eight classes not involving $Q$, the class-by-class
elimination in the proof of \cite[Theorem 4]{DujellaMT}, specialized
to the Fibonacci family as in \cite[Theorem 6]{DujellaMT}, shows that
no integral point can occur outside the class $P+2E_k'(\Q)$.
The possible square or twice-square exceptional Fibonacci terms are
handled there by Cohn's theorem (with the remaining case $k=4$
eliminated modulo $3$); in particular, no exception occurs for odd
$k\geq3$.

Only the local eliminations attached to the eight classes $T$ and $P + T$ are being
reused from the proof of \cite[Theorem 4]{DujellaMT}; those arguments depend on the square-class
equations for the given coset, not on the global rank-one hypothesis used there to
exhaust all cosets.  
It remains to eliminate the eight new classes involving $Q$. 

\begin{lemma}\label{nointegerQ}
None of the eight classes
\[
 Q+T,\qquad Q+P+T,\qquad T\in\{O,A,B,C\},
\]
contains an integer point on $E_k$.
\end{lemma}

\begin{proof}
For $T=A$ or $B$, exactly two coefficients in the corresponding
descent system are negative, hence there can be no integral solution. 
Indeed, since $a,b,c>1$, for integral $x$, all three
$ax + 1$, $bx + 1$, $cx + 1$ are positive
when $x \geq 0$ and all are negative
when $x \leq -1$. 

Consider first the class $Q$.  If an odd prime $p$ occurs to an odd
exponent in $v=F_{k+2}$, then Lemma \ref{coprime}(i) and the first two
equations in the first row of \eqref{table} imply
\[
 p\mid ax+1,\qquad p\mid bx+1.
\]
Hence
\[
 p\mid b(ax+1)-a(bx+1)=b-a=r,
\]
contrary to Lemma \ref{coprime}(i).  Therefore every odd prime occurs
to an even exponent in $F_{k+2}$, and so $F_{k+2}$ is a square or
twice a square, which is impossible.  The same argument applies to
$Q+C$, using the third row of \eqref{table}.

For $Q+P$ and $Q+P+C$, use instead an odd prime occurring to an odd
exponent in $\ell_2=L_{k+2}$ and Lemma \ref{coprime}(ii).  
We conclude that every odd prime occurs to an
even exponent in $L_{k+2}$. Since $k + 2$ is odd, formula (\ref{2adic}) shows that 
$v_2(L_{k+2})$ is
also even. Thus $L_{k+2}$ is a square, contradicting Cohn's theorem because $k + 2 \geq 5$.
\end{proof}

We can now prove the final part of the main theorem.

\begin{proof}[Proof of Theorem \ref{main}(iii)]
By Lemma \ref{nointegerQ}, every integral point must belong to one of
the eight old classes.  As noted above, the elimination in
\cite[Theorem 4]{DujellaMT}, together with its Fibonacci specialization
in \cite[Theorem 6]{DujellaMT}, eliminates all of these except
$P+2E_k'(\Q)$.  By the standard $2$-descent criterion
\cite[Proposition 1]{DujellaMT}, an integral point $(x,y)$ in this
class satisfies
\[
 F_{2k}x+1=\sq,\qquad
 F_{2k+2}x+1=\sq,\qquad
 F_{2k+4}x+1=\sq.
\]
Since $c=F_{2k+4}>1$, an integer satisfying these three square
conditions cannot be negative.  The Hoggatt-Bergum theorem
\cite{DujellaHB}, together with the trivial solution $x=0$, therefore
gives
\[
 x=0,\qquad
 x=4F_{2k+1}F_{2k+2}F_{2k+3}.
\]
Substitution gives precisely \eqref{knownpoints} and \eqref{HBpoint}.
\end{proof}

\section{The generic ranks of the two parity families}\label{generic}

The identity
\[
 L_k^2-5F_k^2=4(-1)^k
\]
suggests separating the odd and even indices already at the level of
one-parameter families.  We now show that the resulting generic ranks
are exactly $2$ and $1$, respectively.

Let $T$ be an indeterminate.  A parametrization of
\[
 \ell^2-5f^2=-4
\]
is
\begin{equation}\label{oddparam}
 f=\frac{T^2-2T+5}{T^2-5},\qquad
 \ell=-\frac{T^2-10T+5}{T^2-5},
\end{equation}
while a parametrization of
\[
 \ell^2-5f^2=4
\]
is
\begin{equation}\label{evenparam}
 f=-\frac{4T}{T^2-5},\qquad
 \ell=-\frac{2(T^2+5)}{T^2-5}.
\end{equation}
In both cases put
\[
 u=\frac{\ell+f}{2},\qquad
 v=\frac{\ell+3f}{2},\qquad
 \ell_1=\frac{\ell+5f}{2},\qquad
 \ell_2=\frac{3\ell+5f}{2},
\]
and
\[
 a=f\ell,\qquad b=u\ell_1,\qquad c=v\ell_2.
\]
We denote by $\mathcal E_-$ and $\mathcal E_+$ the curves
\[
 y^2=(ax+1)(bx+1)(cx+1)
\]
over $\Q(T)$ obtained from \eqref{oddparam} and \eqref{evenparam},
respectively.

For the odd family, the relation $\ell^2-5f^2=-4$, with
$\ell=2u-f$, gives
\[
 u^2-fu-f^2=-1.
\]
Hence the calculation in Proposition \ref{Qpoint} is valid identically
over $\Q(T)$, with $F_{k-1}$ replaced by $u-f$ and
$F_{2k+1}$ by $f^2+u^2$.  Thus $\mathcal E_-$ has the two rational
sections
\[
 \mathcal P=(0,1),\qquad
 \mathcal Q=
 \left(-\frac{u-f}{\ell uv},
       \frac{f^2+u^2}{\ell uv}\right).
\]
The even family has the rational section $\mathcal P=(0,1)$.

We use the injective specialization criterion of Gusi\'c and Tadi\'c
\cite{GusicTadic}.  In the form needed here, let 
\[
 \mathcal E:\quad V^2=(U-e_1)(U-e_2)(U-e_3),
 \qquad e_i\in\Z[T]
\]
be a nonconstant elliptic curve over $\Q(T)$. 
Suppose that the specialization at 
$T=T_0\in\Q$ is nonsingular and that every nonconstant square-free divisor in $\Z[T]$ of each of
\[
 (e_1-e_2)(e_1-e_3),\qquad
 (e_2-e_1)(e_2-e_3),\qquad
 (e_3-e_1)(e_3-e_2)
\]
takes a non-square value in $\Q$. Then the specialization homomorphism at $T_0$ is
injective.

\begin{theorem}\label{genericrank}
The generic Mordell-Weil ranks of the two families are
\[
 \operatorname{rank}\mathcal E_-(\Q(T))=2,\qquad
 \operatorname{rank}\mathcal E_+(\Q(T))=1.
\]
\end{theorem}

\begin{proof}
We first consider $\mathcal E_-$.  Put $d=T^2-5$.  Starting from
$\mathcal E_\pm$, make the changes
\[
 X=abc\,x,\qquad Y=abc\,y,\qquad Z=X+ab.
\]
Then
\[
 Y^2=Z(Z+M)(Z+N),\qquad
 M=a(c-b),\quad N=b(c-a).
\]
Since $M$ and $N$ have denominators dividing $d^4$, the further change
\[
 U=d^4Z,\qquad V=d^6Y
\]
gives an integral model
\[
 V^2=U(U+m_\pm)(U+n_\pm),\qquad
 m_\pm=d^4M,\quad n_\pm=d^4N.
\]
Here $e_1=0$, $e_2=-m_\pm$ and $e_3=-n_\pm$ belong to
$\Z[T]$, as required by the specialization criterion.  Moreover,
\[
 j=256\,\frac{(m_\pm^2-m_\pm n_\pm+n_\pm^2)^3}
 {m_\pm^2n_\pm^2(m_\pm-n_\pm)^2}.
\]
The formulas below show directly that this is nonconstant (for example,
$\deg m_-=8$, $\deg n_-=7$, whereas $\deg m_+=7$, $\deg n_+=8$).

For the odd family, direct simplification gives
\begin{align*}
m_-={}&-(T^2-10T+5)(T^2-2T+5)\\
&\quad{}\times(T^4+4T^3+30T^2+20T+25),\\
n_-={}&16T(T^2+5)(T^4+30T^2+25),
\end{align*}
and
\begin{align*}
n_--m_-={}&(T^2+2T+5)(T^2+10T+5)\\
&\quad{}\times(T^4-4T^3+30T^2-20T+25).
\end{align*}

The displayed factors are irreducible over $\Q$.
We take
\[
 T_0=\frac{85}{38}.
\]
This is the Fibonacci specialization corresponding to $k=17$, since
\[
 \frac{L_{17}-1}{F_{17}-1}=\frac{85}{38}.
\]
The absolute square classes in $\Q^*/\Q^{*2}$ of the three irreducible
factor values of $m_-$ are
\[
 5\cdot3571,\qquad 5\cdot1597,\qquad 73\cdot149\cdot2221,
\]
those of the three factor values of $n_-$ are
\[
 2\cdot5\cdot17\cdot19,\qquad
 3\cdot5\cdot107,\qquad
 7\cdot23\cdot103681,
\]
and those of the three factor values of $n_--m_-$ are
\[
 5\cdot37\cdot113,\qquad
 5\cdot9349,\qquad
 5\cdot13\cdot141961.
\]
In each union of the
irreducible-factor lists occurring in 
$m_-n_-$, $m_-(n_--m_-)$, and $n_-(n_--m_-)$, 
every nonempty product has a nontrivial absolute square class.  Since
$n_-$ has content $16$, the first and third products in condition (A)
also admit the constant square-free factor $2$; adjoining the square
class of $2$ still leaves all nonempty products nontrivial.  Hence, for
either choice of sign, every nonconstant square-free divisor in
condition (A) of \cite{GusicTadic} specializes to a nonsquare in $\Q$.
Therefore the specialization criterion applies and specialization
at $T=85/38$ is injective.

At this value, \eqref{oddparam} gives
\[
 (f,\ell)=(1597,3571)=(F_{17},L_{17}),
\]
so the specialization of $\mathcal E_-$ is $E_{17}$; the integral model
above specializes to a $\Q$-isomorphic curve.  It is known that
\[
 \operatorname{rank}E_{17}(\Q)=2
\]
(see \cite[Section 5]{DujellaMT}).  On the other hand, the sections
$\mathcal P$ and $\mathcal Q$ specialize at $k=17$ to the independent
points $P_{17}$ and $Q_{17}$ by Proposition \ref{independent}.  Thus
\[
 2\leq\operatorname{rank}\mathcal E_-(\Q(T))
 \leq\operatorname{rank}E_{17}(\Q)=2.
\]

For the even family the same construction gives
\[
 V^2=U(U+m_+)(U+n_+),
\]
where
\begin{align*}
m_+={}&16T(T^2+5)
 (T^4+8T^3+30T^2+40T+25),\\
n_+={}&(T^2+2T+5)(T^2+10T+5)\\
&\quad{}\times(3T^4+20T^3+90T^2+100T+75),
\end{align*}
and
\begin{align*}
n_+-m_+={}&(T+1)(T+5)(3T^2+10T+15)\\
&\quad{}\times(T^4+4T^3+30T^2+20T+25).
\end{align*}
Again the displayed factorizations are into irreducible polynomials over
$\Q$.  We take $T_0=20/9$.  The square classes of the irreducible factor
values are
\[
\begin{array}{c|ccc}
m_+&
5&5\cdot7\cdot23&17\cdot53\cdot109\\[1mm]
n_+&
5\cdot233&5\cdot521&3\cdot90481
\end{array}
\]
and
\[
\begin{array}{c|cccc}
n_+-m_+&
29&5\cdot13&3\cdot5\cdot281&3001.
\end{array}
\]
Again, every nonempty product in each of the three relevant unions
has a nontrivial absolute square class.  Since $m_+$ has content $16$,
the first and second products in condition (A) also admit the constant
square-free factor $2$; including it creates no square relation.
Therefore the same is true after either choice of sign, and condition
(A) of \cite{GusicTadic} holds.  Hence specialization is injective.

Formula \eqref{evenparam} at $T=20/9$ gives
\[
 (f,\ell)=(144,322)=(F_{12},L_{12}),
\]
so the specialization of $\mathcal E_+$ is $E_{12}$, and the integral
model specializes to a $\Q$-isomorphic curve.  The computation
\[
 \operatorname{rank}E_{12}(\Q)=1
\]
is recorded in \cite[Section 5]{DujellaMT}.  The section $\mathcal P$ is non-torsion.  Indeed, its specialization
$P_{12}$ is non-torsion by \cite[Theorem 3 and Remark 4]{DujellaMT}.
We conclude
\[
 1\leq\operatorname{rank}\mathcal E_+(\Q(T))
 \leq\operatorname{rank}E_{12}(\Q)=1.
\]
\end{proof}

\section{Rank distribution and a heuristic}

The additional point $Q_k$ gives a simple explanation for a phenomenon
observed computationally in \cite[Remark 4.9.5]{DujellaBook}.
The even and odd subfamilies should be treated separately.

For even $k$, the family has the standard forced point $P_k$, so the
minimal possible rank is $1$.  For odd $k\geq3$, Theorem
\ref{main}(ii) gives two independent forced points, so the minimal
possible rank is $2$.

Computations with Magma for $1\leq k\leq100$ gave
\[
 \#\{k:W(E_k)=+1\}=52,\qquad
 \#\{k:W(E_k)=-1\}=48.
\]
More strikingly, the same split occurs in each parity class:
among the $50$ odd values and among the $50$ even values there are
$26$ curves with root number $+1$ and $24$ with root number $-1$.

Using complete factorizations of the relevant Fibonacci and Lucas numbers, 
we extended the root-number computation to $1\le k\le 709$. The results are
$$
\begin{array}{c|cc}
 & W(E_k)=+1 & W(E_k)=-1\\ \hline
 k\ {\rm even} & 162 & 192\\
 k\ {\rm odd}  & 193 & 162
\end{array}
$$
and hence, in the full range,
$$
\#\{k:W(E_k)=+1\}=355,\qquad
\#\{k:W(E_k)=-1\}=354.
$$
Thus the root numbers are almost perfectly balanced overall, although the
two parity subfamilies exhibit noticeable imbalances in opposite directions.
This is consistent with the expected asymptotic equidistribution, while
suggesting that convergence within the two parity subfamilies may be rather slow.

Combining the parity conjecture with the usual minimalist
philosophy that the rank is normally the smallest one compatible with
the forced points and the root number, one expects
\[
\begin{array}{c|cc}
 & W(E_k)=-1 & W(E_k)=+1\\ \hline
 k\ {\rm even}&1&2\\
 k\ {\rm odd}&3&2.
\end{array}
\]
If the root numbers are asymptotically equidistributed in each parity
subfamily, this leads to the following heuristic.

\begin{conjecture}\label{rankheuristic}
As $k\to\infty$, the curves $E_k$ have ranks $1,2,3$ with densities
\[
 \frac14,\qquad\frac12,\qquad\frac14,
\]
respectively, while the set of $k$ for which
$\operatorname{rank}E_k(\Q)\geq4$ has density $0$.
If, in addition, the contribution of these higher-rank specializations
to the average is negligible, then the average rank should tend to $2$.
\end{conjecture}

The computations for $k\leq100$ are remarkably consistent with this
prediction.  The PARI/GP calculations give
\[
\begin{array}{c|rrrrrrr}
\text{rank interval}&[1,1]&[2,2]&[3,3]&[1,3]&[2,4]&[3,5]&[2,6]\\
\hline
\text{number}&20&45&21&5&6&2&1.
\end{array}
\]
If the actual rank equals the
lower endpoint of every unresolved
interval, the resulting counts are
\[
 25,\qquad52,\qquad23
\]
for ranks $1,2,3$, respectively.  These are exactly the counts predicted
from the computed root numbers and the two forced-rank baselines, with
the exceptional value $k=1$ taken into account.

For completeness, the $2$-Selmer ranks
\[
 \dim_{\mathbf F_2}\operatorname{Sel}_2(E_k)-2
\]
for $1\leq k\leq100$ were distributed as follows:
\[
\begin{array}{c|rrrrrrr}
r_2&1&2&3&4&5&6&7\\ \hline
\#&11&31&33&20&3&1&1.
\end{array}
\]
These data show that the family also has interesting $2$-Selmer
phenomena, but the abundance of odd-$k$ curves expected to have rank
$3$ is already naturally explained by the point $Q_k$, together with
the parity heuristic.

\begin{remark}
Theorem \ref{genericrank} gives a geometric counterpart to the
minimal-rank heuristic: the odd and even Fibonacci specializations lie
on one-parameter families of generic ranks $2$ and $1$, respectively. 
Beyond the parity conjecture, the remaining
heuristic inputs in Conjecture \ref{rankheuristic} are the
expected equidistribution of root numbers, together with the usual
expectation that rank jumps by at least two have density zero.
\end{remark}

\section*{Acknowledgment}

The author acknowledges support from
the Croatian Science Foundation under the
project no. IP-2022-10-5008 (TEBAG),
the project ``Implementation of cutting-edge research and its application as part of the
Scientific Center of Excellence for Quantum and Complex Systems, and Representations of Lie Algebras'', Grant No. PK.1.1.10.0004, co-financed by the European
Union through the European Regional Development Fund -- Competitiveness and
Cohesion Programme 2021--2027,
and the European Union:
NextGenerationEU through the National Recovery and Resilience Plan 2021--2026, Institutional grant of University of Zagreb Faculty of Science (IK IA 1.1.3. Impact4Math).

The additional point $Q_k$ was found in 2026 during an exploratory
investigation by the author, based on PARI/GP and Magma computations,
with assistance from ChatGPT (OpenAI, GPT-5.6 Sol).  All identities,
descent arguments, and computational conclusions used in this paper
were subsequently checked independently.

The author would like to thank Ana Jurasi\'c and Matija Kazalicki for their useful comments and suggestions.


\begin{thebibliography}{99}

\bibitem{Cohn}
J.~H.~E. Cohn,
Lucas and Fibonacci numbers and some Diophantine equations,
\emph{Proc. Glasgow Math. Assoc.} \textbf{7} (1965), 24--28.

\bibitem{DujellaHB}
A. Dujella,
A proof of the Hoggatt--Bergum conjecture,
\emph{Proc. Amer. Math. Soc.} \textbf{127} (1999), 1999--2005.

\bibitem{DujellaParam}
A. Dujella,
A parametric family of elliptic curves,
\emph{Acta Arith.} \textbf{94} (2000), 87--101.

\bibitem{DujellaMT}
A. Dujella,
Diophantine $m$-tuples and elliptic curves,
\emph{J. Th\'eor. Nombres Bordeaux} \textbf{13} (2001), 111--124.

\bibitem{DujellaBook}
A. Dujella,
\emph{Diophantine $m$-tuples and Elliptic Curves},
Developments in Mathematics, vol.~79, Springer, Cham, 2024.

\bibitem{DujellaPetho}
A. Dujella and A. Peth\H{o},
Integer points on a family of elliptic curves,
\emph{Publ. Math. Debrecen} \textbf{56} (2000), 321--335.

\bibitem{FL}
F. Luca and Y. Fujita, 
On Diophantine quadruples of Fibonacci numbers, 
\emph{Glas. Mat. Ser. III} \textbf{52} (2017), 221--234.

\bibitem{FL2}
F. Luca and Y. Fujita, 
There are no Diophantine quadruples of Fibonacci numbers, 
\emph{Acta Arith.} \textbf{185} (2018), 19--38.

\bibitem{GusicTadic}
I. Gusi\'c and P. Tadi\'c,
Injectivity of the specialization homomorphism of elliptic curves,
\emph{J. Number Theory} \textbf{148} (2015), 137--152.

\bibitem{HLT}
B. He, F. Luca and A. Togb\'e, 
Diophantine triples of Fibonacci numbers,
\emph{Acta Arith.} \textbf{175} (2016), 57--70. 

\bibitem{HoggattBergum}
V.~E. Hoggatt, Jr. and G.~E. Bergum,
A problem of Fermat and the Fibonacci sequence,
\emph{Fibonacci Quart.} \textbf{15} (1977), 323--330.

\end{thebibliography}
\end{document}